\documentclass[3p,10pt]{elsarticle}

\usepackage{amssymb,latexsym,amsmath,amsthm,lineno,color}

\journal{}

\newcommand{\set}[1]{\left\{#1\right\}}
\newcommand{\abs}[1]{\left|#1\right|}
\newcommand{\p}{\partial}

\newcommand{\mc}{\mathbf{c}}

\newcommand{\mx}{\mathbf{x}}
\newcommand{\my}{\mathbf{y}}

\newcommand{\mU}{\mathbf{U}}
\newcommand{\mV}{\mathbf{V}}
\newcommand{\mW}{\mathbf{W}}
\newcommand{\vn}{\boldsymbol{\nu}}

\newcommand{\vt}{\boldsymbol{\theta}}
\newcommand{\vv}{\boldsymbol{\vartheta}}
\newcommand{\vx}{\boldsymbol{\xi}}
\newcommand{\vz}{\boldsymbol{\zeta}}

\newtheorem{thm}{Theorem}[section]

\newtheorem{lem}[thm]{Lemma}

\theoremstyle{definition}

\newtheorem{exmp}{Example}[section]

\theoremstyle{remark}

\begin{document}

\begin{frontmatter}



\title{A novel study on subspace migration for imaging of a sound-hard arc}

\author{Won-Kwang Park}
\ead{parkwk@kookmin.ac.kr}
\address{Department of Information Security, Cryptology, and Mathematics, Kookmin University, Seoul, 02707, Korea.}

\begin{abstract}
In this study, the influence of a test vector selection used in subspace migration to reconstruct the shape of a sound-hard arc in a two-dimensional inverse acoustic problem is considered. In particular, a new mathematical structure of imaging function is constructed in terms of the Bessel functions of the order $0$, $1$, and $2$ of the first kind based on the structure of singular vectors linked to the nonzero singular values of a Multi-Static Response (MSR) matrix. This structure indicates that imaging performance of subspace migration is highly related to the unknown shape of arc. The simulation results with noisy data indicate support for the derived structure.
\end{abstract}


\end{frontmatter}





\section{Introduction}
A pioneering study \cite{M1} focused on investigating an inverse acoustic problem from a sound-hard arc (a perfectly conducting crack in Transverse Electric mode for an electromagnetic inverse scattering problem) in two-dimensions. In the aforementioned study, a Newton-type iterative method is investigated to retrieve the shape of a single smooth arc. This is followed by suggesting and applying various techniques to retrieve the sound-hard arcs. Examples of these techniques include an inverse Fourier transform \cite{AS1}, hybrid method \cite{KS1}, Newton's method \cite{L1}, two-step method \cite{L2}, Multiple Signal Classification (MUSIC) \cite{PL1}, and subspace migration \cite{P-SUB3}.

Extant studies have demonstrated the feasibilities of Newton-type iteration schemes to retrieve arcs. However, it is very difficult to extend the methods shown to multiple arcs as it requires laborious calculation of the Fr{\'e}chet derivative, which require high computation time  and \textit{a priori} information of the unknown arc. Conversely, non-iterative schemes, such as MUSIC and subspace migration, were applied to the retrieving multiple arcs. Unfortunately, for a successful retrieve, it is necessary to apply appropriate test vectors, and a unit outward normal vector along the arc must be known to select the test vectors. This requires the application of a set of directions to adjust the normal vectors correspondingly; moreover, large computational costs are entailed (see \cite{PL1,P-SUB3,HSZ1} for instance). Therefore, inappropriate test vectors were applied instead of the application of an appropriate test vector. Nevertheless, theoretical reasons of certain phenomena with the application of appropriate and inappropriate test vectors are not yet explained to date. Thus, the present study focuses on exploring the structure of an imaging function adopted in subspace migration.

In this study, the structure of an imaging function in subspace migration is analyzed by establishing a relationship with Bessel functions of the order $0$, $1$, and $2$ of the first kind. This is based on the physical factorization of a Multi-Static Response matrix (see \cite{HSZ1} for instance). The analyzed structure indicated that the retrieval of a sound-hard arc is highly dependent on the selection of test vectors, and it is necessary to adjust the test vectors by normal vectors to obtain an optimal result. In addition, the results of the numerical simulation exhibited support for the analyzed structure.

The rest of this paper is organized as follows. Section \ref{sec:2} introduces a two-dimensional direct scattering problem in the presence of a sound-hard arc. Section \ref{sec:3} illustrates the subspace migration based an imaging algorithm and derives a relationship with a Bessel function of an integer order of the first kind. In Section \ref{sec:4}, the simulation results are demonstrated and discussed to verify our theoretical results. Finally, the conclusions are presented in Section \ref{sec:5}.

\section{Two-dimensional direct scattering problem}\label{sec:2}
In this study, the two-dimensional direct scattering of acoustic waves is introduced by an open sound-hard arc denoted by $\Gamma$. Following \cite{M1}, it is assumed that $\Gamma$ is an oriented piecewise-smooth non-intersecting arc without a cusp that can be represented as follows:
\[\Gamma:=\set{\boldsymbol{\gamma}(s):-1\leq s\leq1},\]
where $\boldsymbol{\gamma}$ denotes an injective $\mathcal{C}^3$ function. In this study, the plane-wave illumination is considered, i.e., incident wave field is selected as follows:
\[\psi_{\mathrm{inc}}(\mx;\vt):=e^{ik\vt\cdot\mx},\]
where a two-dimensional unit vector $\vt$ denotes the propagation direction. Let $\psi(\mx;\vt)\in\mathcal{C}^2(\mathbb{R}^2\backslash\Gamma)\cap\mathcal{C}(\overline{\mathbb{R}^2\backslash\Gamma})$ be the total wave that satisfies the following Helmholtz equation:
\[\triangle \psi(\mx;\vt)+k^2\psi(\mx;\vt)=0\quad\mbox{in}\quad\mathbb{R}^2\backslash\Gamma\]
with the following Neumann boundary condition:
\[\frac{\p \psi(\mx;\vt)}{\p\vn(\mx)}=0\quad\mbox{on}\quad\Gamma\backslash\set{\boldsymbol{\gamma}(-1),\boldsymbol{\gamma}(1)},\]
where $\vn(\mx)$ denotes a unit normal to $\Gamma$ at $\mx$ and $k=2\pi/\lambda$ denotes a positive wavenumber, where $\lambda$ denotes the given wavelength. In this paper, we consider the imaging of extended sound-hard arc so, assume that $\lambda$ satisfies $\lambda\ll\mbox{length of }\Gamma$.

As widely known, $\psi(\mx;\vt)$ can be decomposed as follows: $\psi(\mx;\vt)=\psi_{\mathrm{inc}}(\mx;\vt)+\psi_{\mathrm{scat}}(\mx;\vt)$, where $\psi_{\mathrm{scat}}(\mx;\vt)$ denotes the scattered wave fields that satisfies the Sommerfeld radiation condition
\[\lim_{\abs{\mx}\to\infty}\sqrt{\abs{\mx}}\left(\frac{\p \psi_{\mathrm{scat}}(\mx;\vt)}{\p\abs{\mx}}-ik\psi_{\mathrm{scat}}(\mx;\vt)\right)=0\]
uniformly in all directions $\vv=\mx/\abs{\mx}$.
Based on \cite{M2}, $\psi_{\mathrm{scat}}(\mx;\vt)$ can be expressed as the following double-layer potential with (unknown) density function $\varphi(\mx,\vt)$:
\[\psi_{\mathrm{scat}}(\mx;\vt)=\int_{\Gamma}\frac{\p\Phi(\mx,\my)}{\p\vn(\my)}\varphi(\my,\vt)d\my\quad\mbox{for}\quad\mx\in\mathbb{R}^2\backslash\Gamma,\]
where $\Phi(\mx,\my)$ denotes the two-dimensional fundamental solution to the Helmholtz equation as follows:
\[\Phi(\mx,\my):=-\frac{i}{4}\mathrm{H}_0^1(k\abs{\mx-\my})\quad\mbox{for}\quad\mx\ne\my.\]
where  $\mathrm{H}_0^1$ denotes the Hankel function of order zero and of the first kind.

The far-field pattern $\psi_{\infty}(\vv,\vt)$ of the scattered wave $\psi_{\mathrm{scat}}(\mx,\vt)$ is defined on a two-dimensional unit circle. It can be represented as follows:
\[\psi_{\mathrm{scat}}(\mx,\vt)=\frac{e^{ik\abs{\mx}}}{\sqrt{\abs{\mx}}}\left\{\psi_{\infty}(\vv,\vt)+O\left(\frac{1}{\abs{\mx}}\right)\right\}\]
that is, uniformly in all directions $\vv=\mx/\abs{\mx}$ and $\abs{\mx}\longrightarrow\infty$. Following the formulation in an extant study by \cite{M2},  $\psi_\infty(\vv,\vt)$ can be represented as follows:
\begin{equation}\label{FFP}
\psi_\infty(\vv,\vt)=-\sqrt{\frac{k}{8\pi}}e^{-i\frac{\pi}{4}}\int_{\Gamma}(\vv\cdot\vn(\my))e^{-ik\vv\cdot\my}\varphi(\my,\vt)d\my.
\end{equation}

\section{Subspace migration imaging: method and mathematical structure}\label{sec:3}
In this section,  the subspace migration imaging function based on the far-field pattern (\ref{FFP}) is briefly introduced. Prior to commencing the derivation, it is assumed that there exists only a single arc $\Gamma$, and the arc is divided into $M$ different segments of sizes of the order that corresponds to half the wavelength $\lambda/2$. Thus, based on the resolution limit, only a single point denoted as $\my_m\in\Gamma$ for $m=1,2,\cdots,M$ at each segment can be mapped (see \cite{ABC} for instance).

For the purposes of simplicity, the following is assumed: we have $N(>M)$-different number of incident and corresponding observation directions $\vt_l$ and $\vv_j$, respectively for $j,l=1,2,\cdots,N$. Since the full-view inverse problem, $\vt_n$ is set as follows:
\[\vt_n=-\left[\cos\frac{2\pi(n-1)}{N},\sin\frac{2\pi(n-1)}{N}\right]^T.\]
Thus, the Multi-Static Response (MSR) matrix is as follows:
$\mathbb{K}:=[\psi_\infty(\vv_j,\vt_l;k)]_{j,l=1}^{N}$. If the incident and observation directions coincide, i.e., if $\vv_j=-\vt_j$, then $\mathbb{K}$ can be expressed as follows:
\[\mathbb{K}=\left[\begin{array}{cccc}
\psi_{\infty}(\vv_1,\vt_1) & \psi_{\infty}(\vv_1,\vt_2) & \cdots & \psi_{\infty}(\vv_1,\vt_N)\\
\psi_{\infty}(\vv_2,\vt_1) & \psi_{\infty}(\vv_2,\vt_2) & \cdots & \psi_{\infty}(\vv_2,\vt_N)\\
\vdots&\vdots&\ddots&\vdots\\
\psi_{\infty}(\vv_N,\vt_1) & \psi_{\infty}(\vv_N,\vt_2) & \cdots & \psi_{\infty}(\vv_N,\vt_N)\\
\end{array}\right]=\sqrt{\frac{k}{8\pi}}e^{-i\frac{\pi}{4}}\int_{\Gamma}\mathbb{E}(\vv,\my)^T\mathbb{F}(\my,\vt)d\my,\]
where $\mathbb{E}(\vv,\my)$ denotes the illumination vector
\begin{align}
\begin{aligned}\label{VecEN}
\mathbb{E}(\vv,\my)&=-\bigg[(\vv_1\cdot\vn(\my))e^{-ik\vv_1\cdot \my},\cdots,(\vv_N\cdot\vn(\my))e^{-ik\vv_N\cdot\my}\bigg]=\bigg[(\vt_1\cdot\vn(\my))e^{ik\vt_1\cdot\my},\cdots,(\vt_N\cdot\vn(\my))e^{ik\vt_N\cdot\my}\bigg]
\end{aligned}
\end{align}
and where $\mathbb{F}(\vv,\my)$ denotes the corresponding density vector.
\[\mathbb{F}(\my,\vt)=\bigg[\varphi(\my,\vt_1),\varphi(\my,\vt_2),\cdots,\varphi(\my,\vt_N)\bigg].\]
It should be noted that the range of $\mathbb{K}$ is determined by the span of the $\mathbb{E}(\vv,\my)$ corresponding to $\Gamma$. This indicates that the signal subspace can be determined by selecting the first $M-$singular vectors of $\mathbb{K}$. A further discussion can be found in a study by \cite{HSZ1}.

Based on the abovementioned observation, the imaging function can be introduced as follows in which  a Singular Value Decomposition (SVD) of $\mathbb{K}$ is performed as follows:
\begin{equation}\label{SVD}
  \mathbb{K}=\mathbb{USV}^*=\sum_{n=1}^{N}\sigma_n\mU_n\mV_n^*\approx\sum_{n=1}^{M}\sigma_n\mU_n\mV_n^*,
\end{equation}
where $\sigma_n$ denotes nonzero singular values, and $\mU_n$ and $\mV_n$ denote left- and right-singular vectors of $\mathbb{K}$, respectively. Following the structure of (\ref{VecEN}), a test vector is introduced as
\begin{equation}\label{TestVector}
\mW(\mx)=\sqrt{\frac{2}{N}}\bigg[(\vt_1\cdot\mc_1)e^{ik\vt_1\cdot\mx},\cdots,(\vt_N\cdot\mc_N)e^{ik\vt_N\cdot\mx}\bigg]^T.
\end{equation}
The subspace migration imaging function can then be designed as
\[\mathfrak{F}(\mx)=\abs{\sum_{m=1}^{M} \left(\mW(\mx)^*\mU_m\right)\left(\mW(\mx)^*\overline{\mV}_m\right)}.\]
It should be noted that $\mathfrak{F}(\mx)$ exhibits peaks of magnitude of $1$ when $\mx=\my_m\in\Gamma$ and $\mc_n=\vn(\my_m)$ based on the orthonormal property of singular vectors. Otherwise, it exhibits small magnitudes.

In order to guarantee a good imaging performance, it is necessary to consider the proper selection of $\mc_n$, $n=1,2,\cdots,N$, of (\ref{TestVector}). Following \cite{PL1}, it is essential that $\mc_n$ is of the form $\vn(\my_m)$ for $m=1,2,\cdots,M$. However, there is no \textit{a priori} information of the shape of $\Gamma$, and a set of directions is applied instead of $\mc_n$ or a fixed test vector $\vx$ is applied based on a previous study by \cite{PL1,P-SUB3,HSZ1}. However, there are limited theoretical studies on mathematical theories corresponding to the dependency of test vector selection. Hence, a mathematical structure of the subspace migration imaging function is identified. For this, specific identities that play a key role of investigation of structure are introduced. Details are provided in the proof given in the Appendix \ref{sec:A}.

\begin{lem}\label{Lemma}
  With respect to a sufficiently large $N$, $\vx,\vz\in\mathbb{S}^1$, and $\mx\in\mathbb{R}^2$, the following expression holds  \begin{enumerate}
  \item if $\vx\ne\vz$ then
  \begin{align}
  \begin{aligned}\label{Identity1}
    \frac{1}{N}\sum_{n=1}^{N}(\vt_n\cdot\vx)(\vt_n\cdot\vz)e^{ik\vt_n\cdot\mx}&=\frac{1}{2\pi}\int_{\mathbb{S}^1}(\vt\cdot\vx)(\vt\cdot\vz)e^{ik\vt\cdot\mx}d\vt\\
    &=\frac12(\vx\cdot\vz)\bigg(J_0(k|\mx|)-J_2(k|\mx|)\bigg)-\left(\frac{\mx}{|\mx|}\cdot\vx\right)\left(\frac{\mx}{|\mx|}\cdot\vz\right)J_2(k|\mx|).
  \end{aligned}
  \end{align}
  \item if $\vx=\vz$ then
  \begin{equation}\label{Identity2}
  \frac{1}{N}\sum_{n=1}^{N}(\vt_n\cdot\vx)(\vt_n\cdot\vz)e^{ik\vt_n\cdot\mx}=\int_{\mathbb{S}^1}(\vt\cdot\vx)^2e^{ik\vt\cdot\mx}d\vt=\frac{1}{2}J_0(k|\mx|).
  \end{equation}
  \end{enumerate}
\end{lem}

This is followed by establishing a mathematical structure of the imaging function $\mathfrak{F}(\mx)$ as follows:
\begin{thm}\label{TheoremStructure} It is assumed that the total number of directions $N>M$ is sufficiently large. Then, $\mathfrak{F}(\mx)$ can be represented as follows:
\begin{enumerate}
\item If $\mc_n\approx\vn(\my_m)$ then,
\begin{equation}\label{structure1}
  \mathfrak{F}(\mx)\approx\sum_{m=1}^{M}J_0(k|\mx-\my_m|)^2.
\end{equation}
\item If $\mc_n=\vt_n$ for all $n$ then,
  \begin{equation}\label{structure2}
  \mathfrak{F}(\mx)\approx2\sum_{m=1}^{M}\left(\frac{\mx-\my_m}{|\mx-\my_m|}\cdot\vn(\mx_m)\right)^2J_1(k|\mx-\my_m|)^2.
  \end{equation}
  \item If $\mc_n\equiv\vx$ and neither $\vx=\vn(\my_m)$ nor $\vx=\vt_n$ then,
\begin{multline}\label{structure3}
\mathfrak{F}(\mx)\approx\sum_{m=1}^{M}\bigg\{(\vn(\my_m)\cdot\vx)\bigg(J_0(k|\mx-\my_m|)-J_2(k|\mx-\my_m|)\bigg)\\
    -2\left(\frac{\mx-\my_m}{|\mx-\my_m|}\cdot\vn(\my_m)\right)\left(\frac{\mx-\my_m}{|\mx-\my_m|}\cdot\vx\right)J_2(k|\mx-\my_m|)\bigg\}^2.
\end{multline}
\end{enumerate}
\end{thm}
\begin{proof}
Based on a study by \cite{AGKPS}, $\mU_m\approx\mW(\my_m)$ and $\overline{\mV}_m\approx\mW(\my_m)$ for $m=1,2,\cdots,M$.
\begin{enumerate}
  \item It is assumed that $\mc_n\approx\vn(\my_m)$. Then, Lemma \ref{Lemma} is applied to derive the following:
      \begin{align*}
(\mW(\mx)^*\mU_m)(\mW(\mx)^*\overline{\mV}_m)&=\left(\mW(\mx)^*\mW(\my_m)\right)\left(\mW(\mx)^*\mW(\my_m)\right)\\
&=\left(\frac{2}{N}\sum_{n=1}^{N}(\vt_n\cdot\vn(\my_m))^2e^{ik\vt_n\cdot(\mx-\my_m)}\right)^2\approx\left(2\int_{\mathbb{S}^1}(\vt\cdot\vn(\my_m))^2e^{ik\vt\cdot(\mx-\my_m)}d\vt\right)^2\\
&=J_0(k|\mx-\my_m|)^2.
\end{align*}
Hence (\ref{structure1}) can be derived.
  \item The derivation of (\ref{structure2}) can be found in a study by \cite{P-SUB3}.
  \item It is assumed that $\mc_n\equiv\vx$ and neither $\vx=\vn(\my_m)$ nor $\vx=\vt_n$. Thus, in a manner similar to the derivation of (\ref{structure1}), the following expression can be evaluated by applying Lemma \ref{Lemma} as follows:
\begin{align*}
(&\mW(\mx)^*\mU_m)(\mW(\mx)^*\overline{\mV}_m)=\left(\mW(\mx)^*\mW(\my_m)\right)\left(\mW(\mx)^*\mW(\my_m)\right)\\
&=\left(\frac{2}{N}\sum_{n=1}^{N}(\vt_n\cdot\vn(\my_m))(\vt_n\cdot\vx)e^{ik\vt_n\cdot(\mx-\my_m)}\right)^2\approx\left(2\int_{\mathbb{S}^1}(\vt\cdot\vn(\my_m))(\vt\cdot\vx)e^{ik\vt\cdot(\mx-\my_m)}d\vt\right)^2\\
&=\bigg\{(\vn(\my_m)\cdot\vx)\bigg(J_0(k|\mx-\my_m|)-J_2(k|\mx-\my_m|)\bigg)-2\left(\frac{\mx-\my_m}{|\mx-\my_m|}\cdot\vn(\my_m)\right)\left(\frac{\mx-\my_m}{|\mx-\my_m|}\cdot\vx\right)J_2(k|\mx-\my_m|)\bigg\}^2.
\end{align*}
Hence, (\ref{structure3}) is derived.
\end{enumerate}
This completes the proof.
\end{proof}

Based on the result in Theorem \ref{TheoremStructure}, specific properties are investigated, and this is summarized as follows:
\begin{enumerate}
  \item If \textit{a priori} information with respect to a unit in the outward normal directions along the arc is available, then an almost true shape of the unit can be retrieved via subspace migration. It is expected that the application of a technique developed in \cite{CZ} allows the selection of optimal vectors $\mc_n$.
  \item If $N$ is small or $N<M$, some unexpected replicas will appear or it is impossible to identify the shape of arc. Related works can be found in \cite{P-SUB7,P-MUSIC1,PL3}.
  \item On the basis of the relationship between the imaging function and Bessel functions, the imaging result is highly dependent on the applied wavenumber $k$. If $k$ is small, one will obtain a result of poor resolution. Otherwise, if $k$ is large, an image of good resolution will be appear but due to the oscillating property of Bessel functions, unexpected artifacts will appear also.
  \item Based on (\ref{structure3}), it is observed that the existence of factor $(\vn(\my_m)\cdot\vx)J_0(k|\mx-\my_m|)$ allows the shape of the arc to be imaged via subspace migration and that the recognition of arc is highly dependent on the selection of $\vx$.
  \item If the selected test vector $\vx$ is orthogonal to $\vn(\my_m)$, it is not possible to retrieve the shape of the arc. In contrast, if $\vx$ is parallel to $\vn(\my_m)$, then the complete shape of the arc can be identified. This supports the traditional results related to sound-hard arcs in inverse scattering/acoustic problems. Furthermore, an optimal result can be obtained if $\vx\cdot\vn(\my_m)$ is close to $1$. Otherwise, it is impossible to recognize the shape of an arc if $\vx\cdot\vn(\my_m)$ is close to $0$.
  \item If the shape of an arc is straight line, it will be possible to retrieve its complete shape because $\vn(\my)$, $\my\in\Gamma$, is constant. However, if the shape of an arc is no more straight line, it is impossible to identify its complete shape with the selection of $\mc_n\equiv\vx$. Therefore, we can conclude that the imaging performance of subspace migration is highly related to the selection of test vectors corresponding to the unit outward vectors on the arc, i.e., unknown shape of arc.
\end{enumerate}

\section{Simulation results}\label{sec:4}
In this section, a few results of the numerical simulation with noisy data are exhibited to support the results in Theorem \ref{TheoremStructure}. Thus, elements  $\psi_{\infty}(\vv_j,\vt_l)$, $j,l=1,2,\cdots,N$, of the $\mathbb{K}$ are generated by solving a second-kind Fredholm integral equation along the arc introduced in \cite[Chapter 3]{N}. A $20$dB white Gaussian random noise is added through MATLAB subroutine \texttt{awgn} in every example illustrated below.

\begin{exmp}[Straight line arc]
  In this example, the reconstruction of a straight line shaped arc is considered as follows:
  \[\Gamma=\set{[s,0.3]^T:-0.5\leq s\leq0.5}.\]
  Figure \ref{Result1} shows the maps of $\mathfrak{F}(\mx)$ with $N=20$ and $\lambda=0.4$ with $\vx$ varied when the arc corresponds to $\Gamma$. It should be noted that an optimal result can be retrieved by applying $\vx=[0,1]^T$ since $\vn(\my_m)=[0,1]^T$ for all $m$. In contrast, a very poor result is derived when the applied $\vx$ is orthogonal to $\vn(\my_m)$. It should be noted that the shape of $\Gamma$ can be identified when $\vx=[\cos(\pi/3),\sin(\pi/3)]^T$ because $\vx$ is close to $\vn(\my_m)$. With respect to $\vx=[\cos(\pi/4),\sin(\pi/4)]^T$, the shape of $\Gamma$ can be identified but instead of a true shape, two lines with large magnitude are more visible because the effect of $J_2(k|\mx-\my_m|)$ is stronger than that of $J_0(k|\mx-\my_m|)$. Correspondingly, the shape of the arc becomes invisible when $\vx=[\cos(\pi/6),\sin(\pi/6)]^T$ since $\vx$ is not close to $\vn(\my_m)$. If $\vx=\vt_n$, two ghost replicas with large magnitudes appear in the neighborhood of the arc as opposed to the true shape such that the true shape of $\Gamma$ can be predicted.
\end{exmp}

\begin{figure}[h]
\begin{center}
\includegraphics[width=0.325\textwidth]{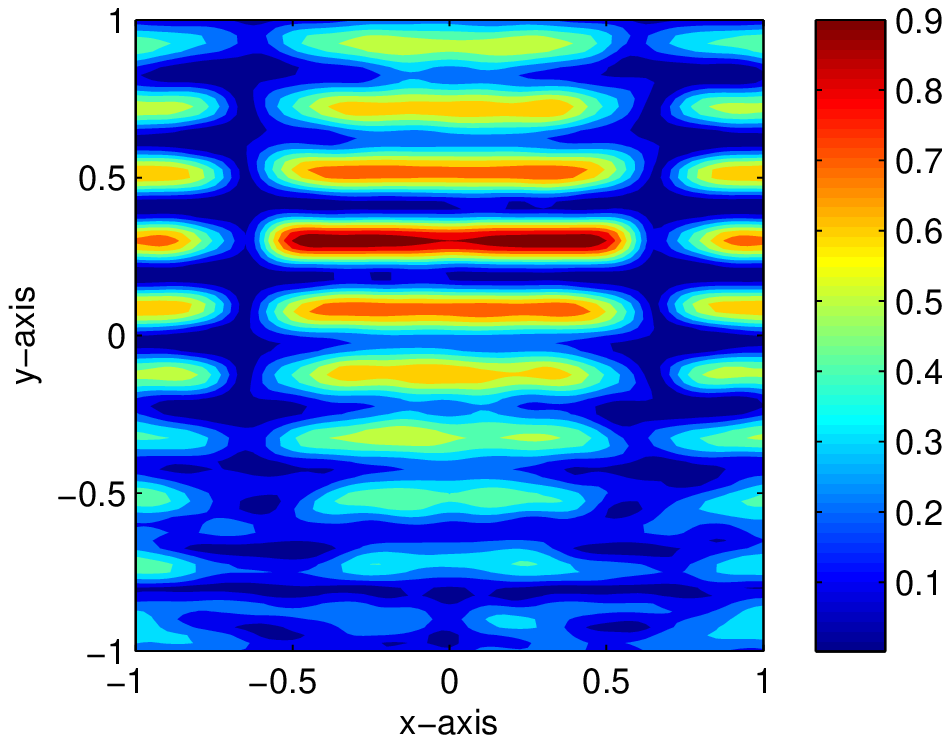}
\includegraphics[width=0.325\textwidth]{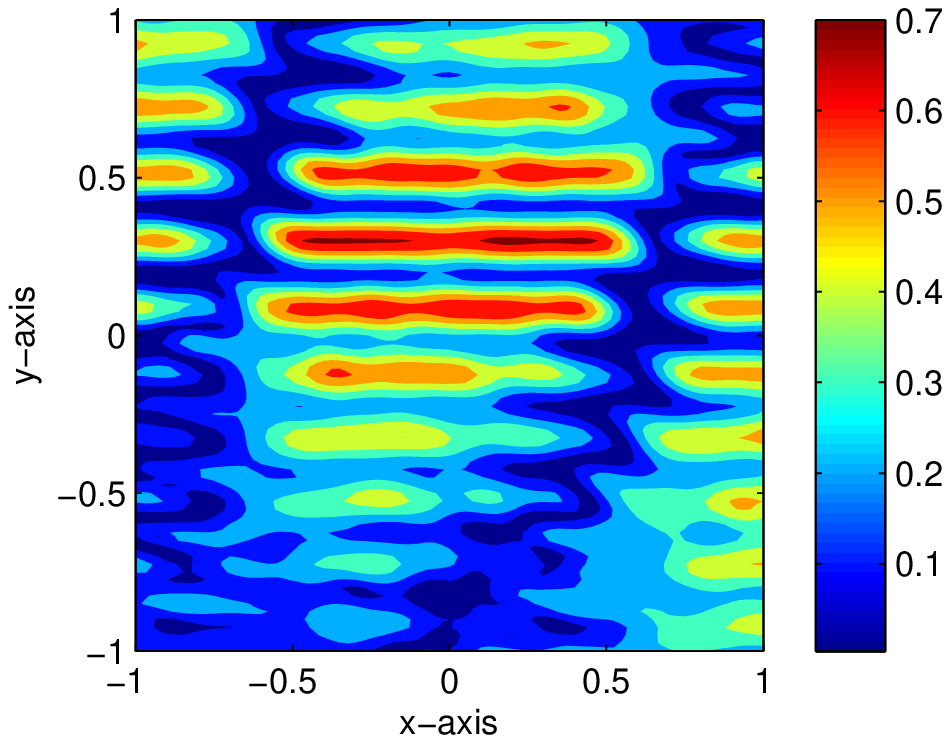}
\includegraphics[width=0.325\textwidth]{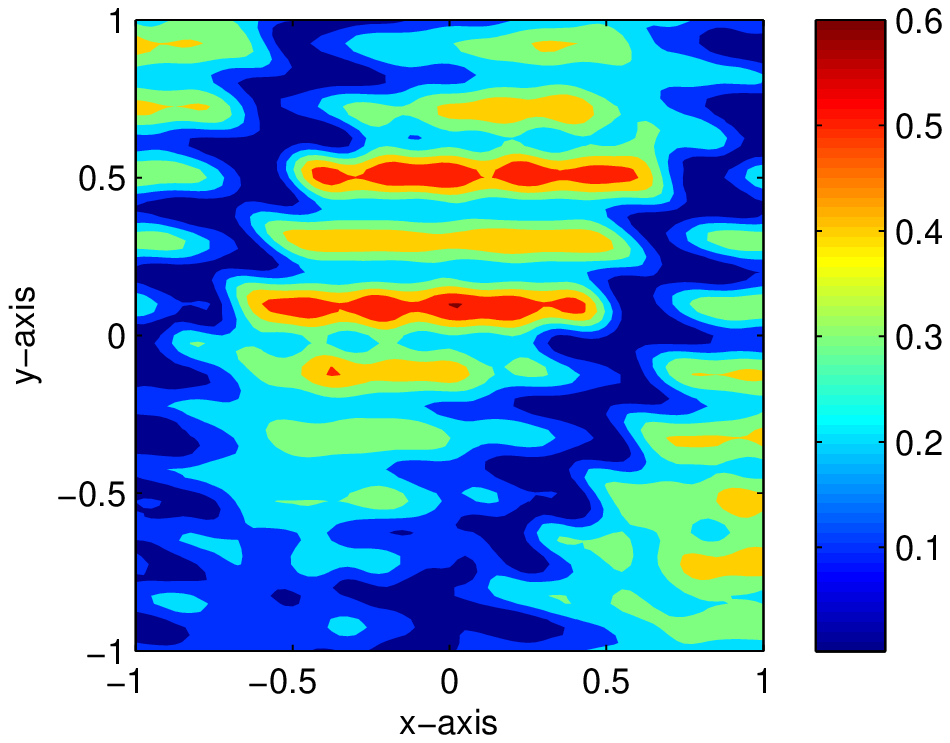}\\
\includegraphics[width=0.325\textwidth]{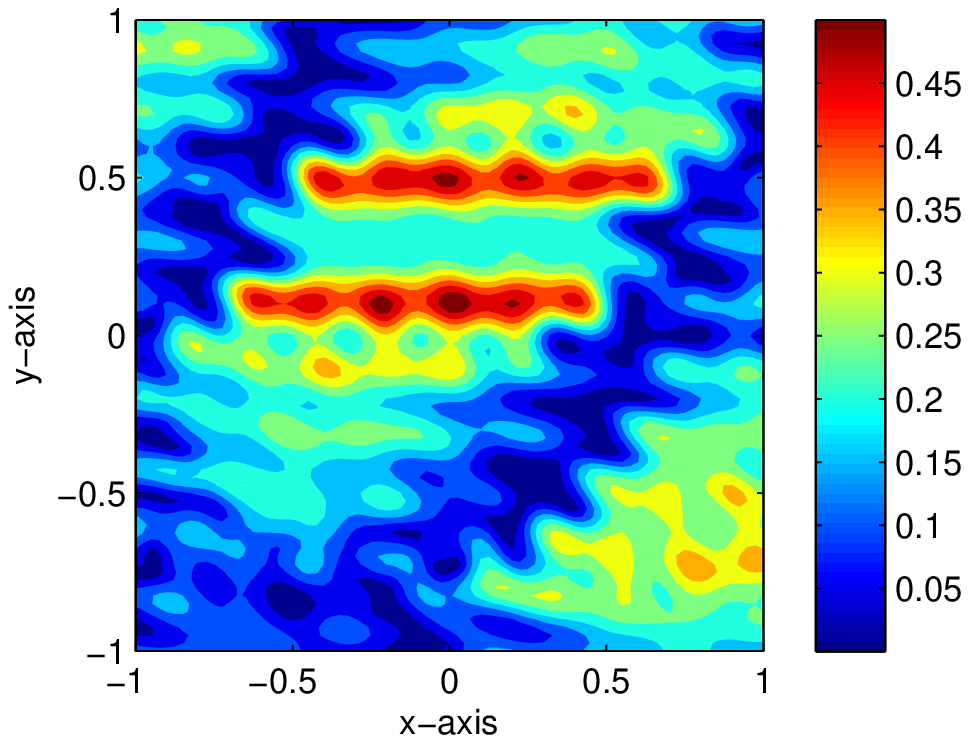}
\includegraphics[width=0.325\textwidth]{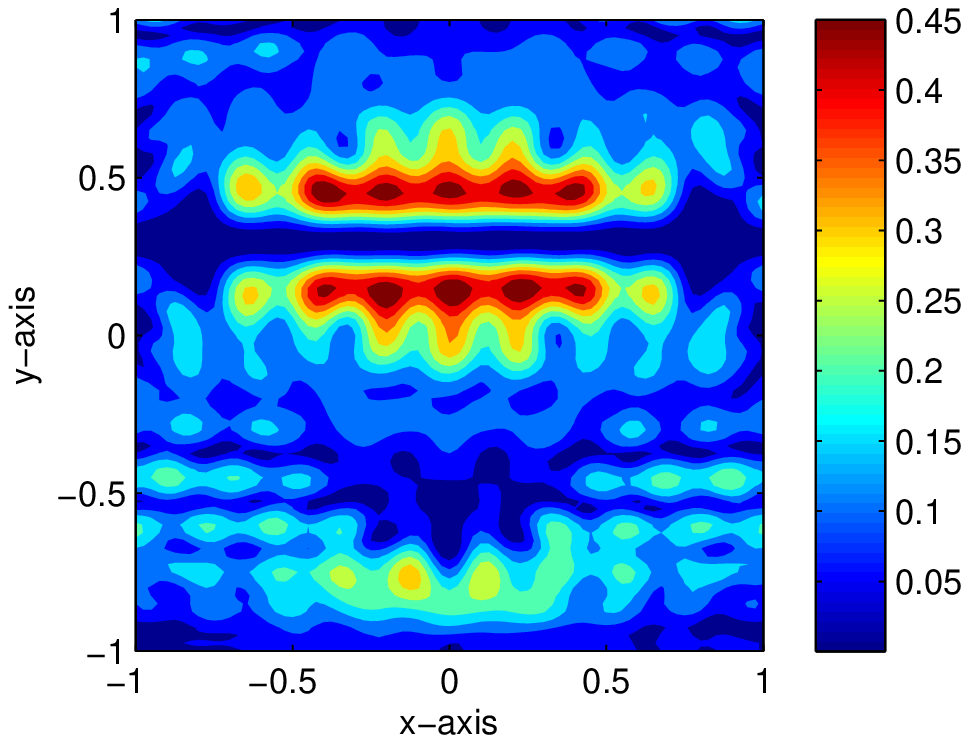}
\includegraphics[width=0.325\textwidth]{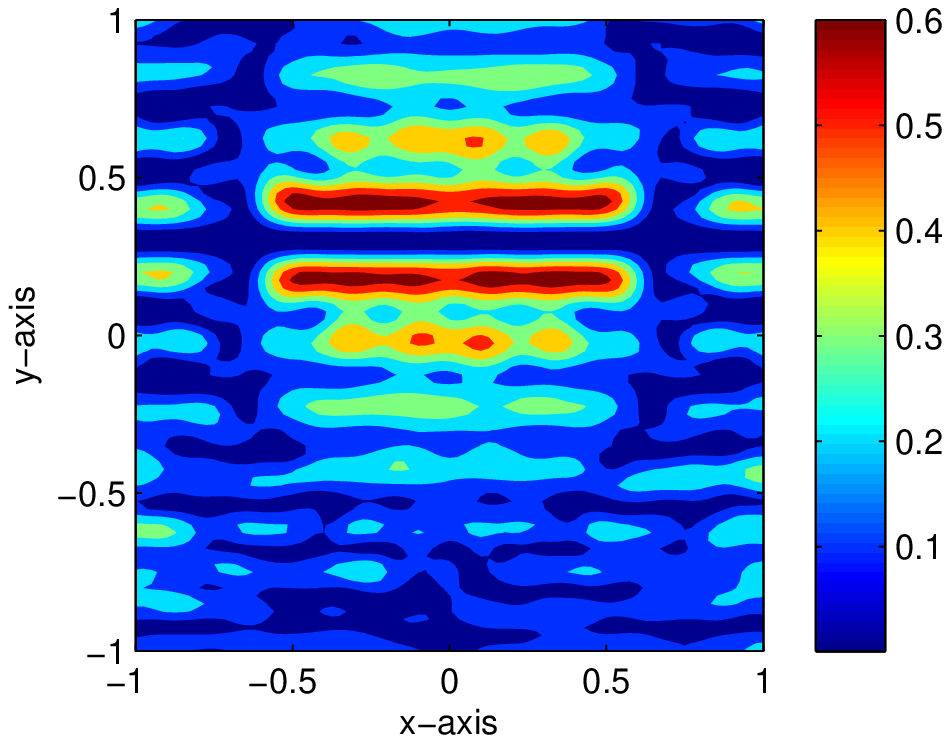}
\caption{\label{Result1}Maps of $\mathfrak{F}(\mx)$ for $\vx=[0,1]^T$ (top, left), $\vx=[\cos(\pi/3),\sin(\pi/3)]^T$ (top, center), $\vx=[\cos(\pi/4),\sin(\pi/4)]^T$ (top, right), $\vx=[\cos(\pi/6),\sin(\pi/6)]^T$ (bottom, left), $\vx=[1,0]^T$ (bottom, center), and $\mc_n\cdot\vt_n=1$ (bottom, right) when the arc is $\Gamma$.}
\end{center}
\end{figure}

\begin{exmp}[Curve-like arc]
Figure \ref{Result2} shows the maps of $\mathfrak{F}(\mx)$ with various $\vx$ for an complex shaped arc
\[\Gamma=\set{[s,0.5\cos0.5\pi s+0.2\sin0.5\pi s-0.1\cos1.5\pi s]^T:-1\leq s\leq1}.\]
For this, $N=32$ and $\lambda=0.5$ are applied to generate a MSR matrix. Moreover, $\vn(\my_m)$ varies along the arc since $\Gamma_2$ is not a straight line. This implies that it is not possible to obtain good results by applying fixed vectors. Fortunately, an outline shape of $\Gamma$ can be recognized if $\vx=[0,1]^T$. Furthermore, the application of $\vx=\vt_n$ allows the prediction of the true shape of $\Gamma_2$ by considering two replicas with large magnitudes.
\end{exmp}

\begin{figure}[h]
\begin{center}
\includegraphics[width=0.325\textwidth]{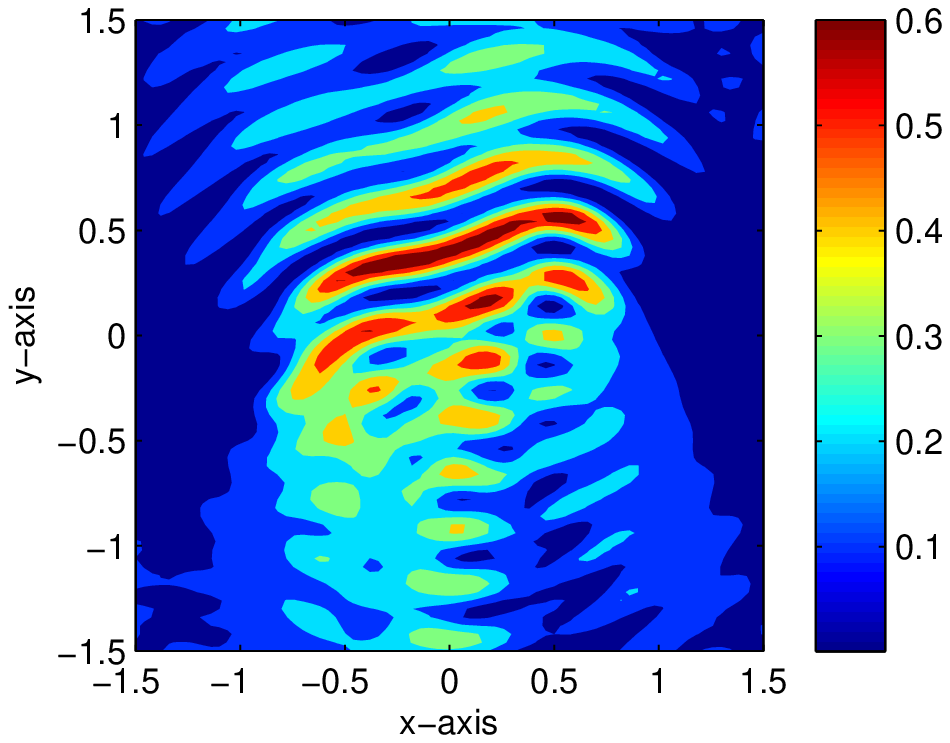}
\includegraphics[width=0.325\textwidth]{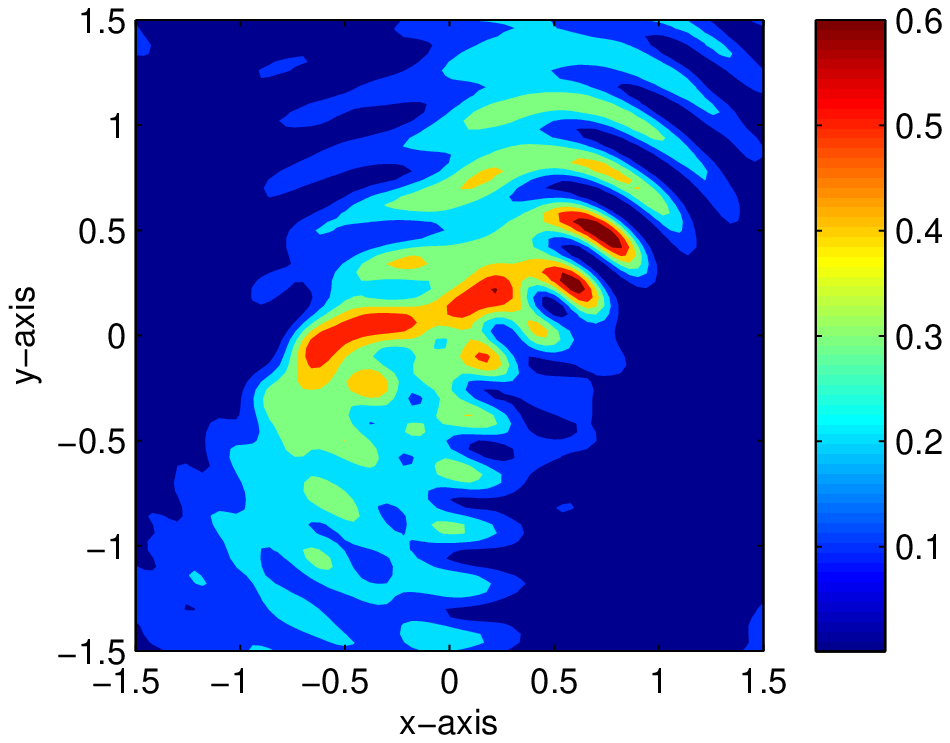}
\includegraphics[width=0.325\textwidth]{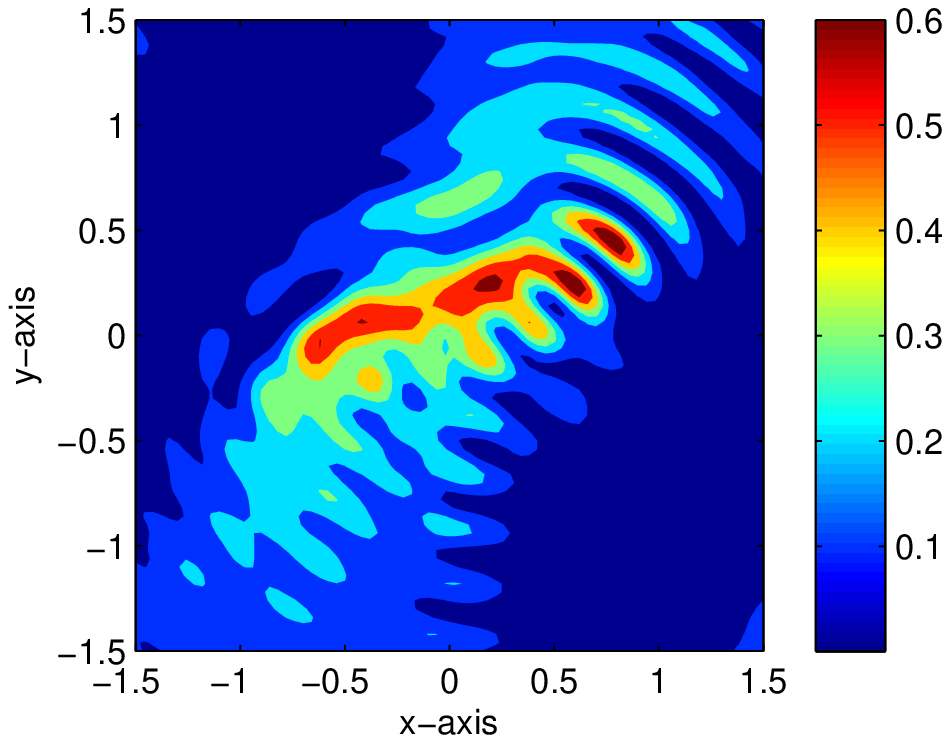}\\
\includegraphics[width=0.325\textwidth]{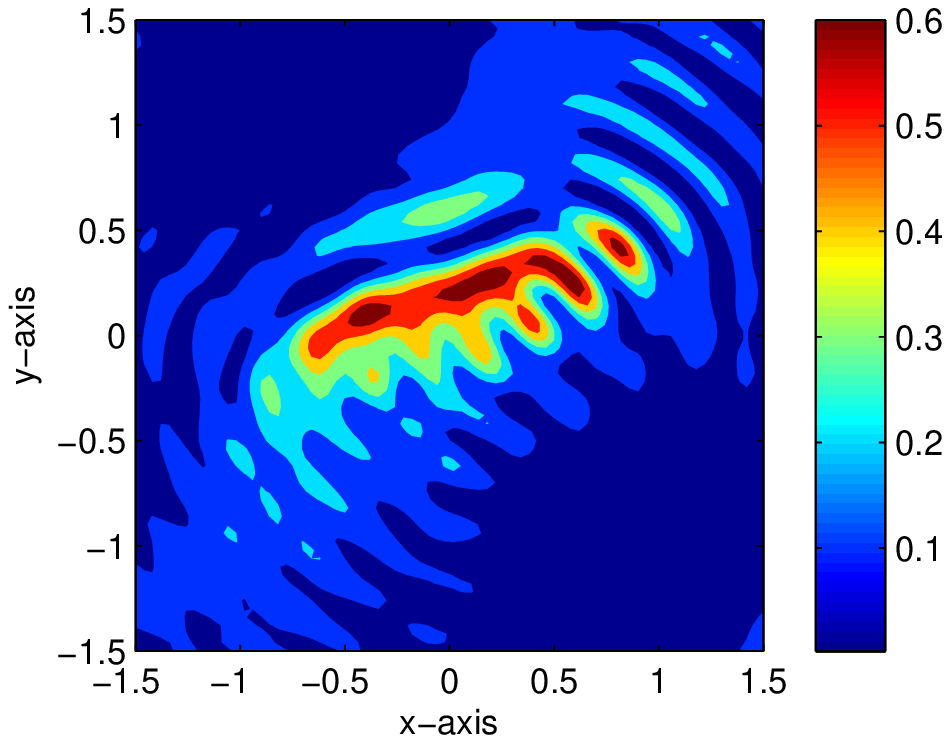}
\includegraphics[width=0.325\textwidth]{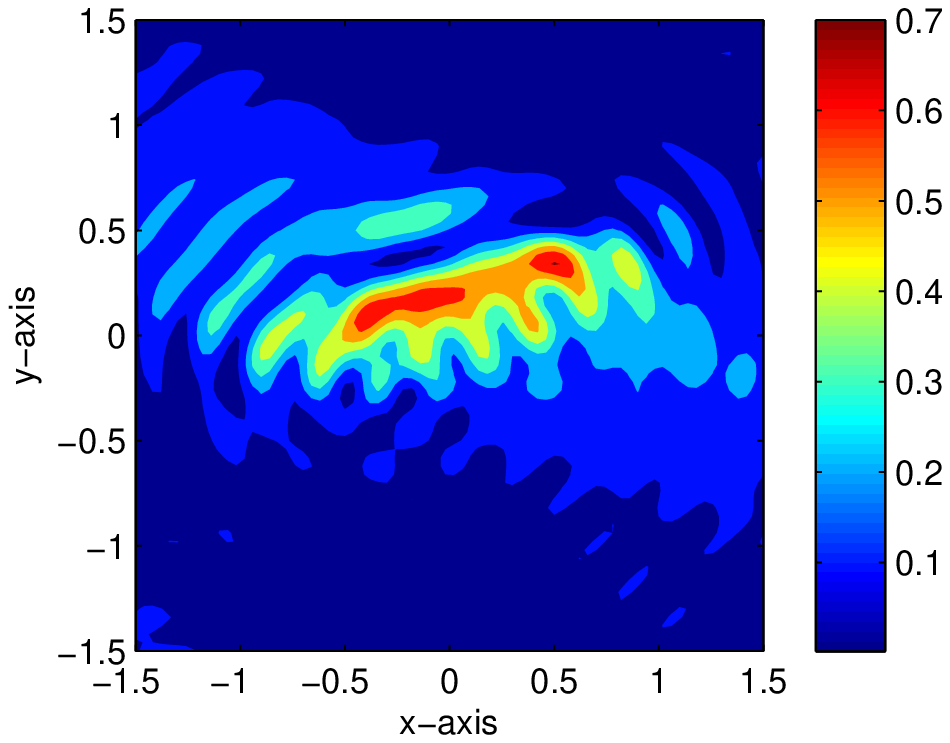}
\includegraphics[width=0.325\textwidth]{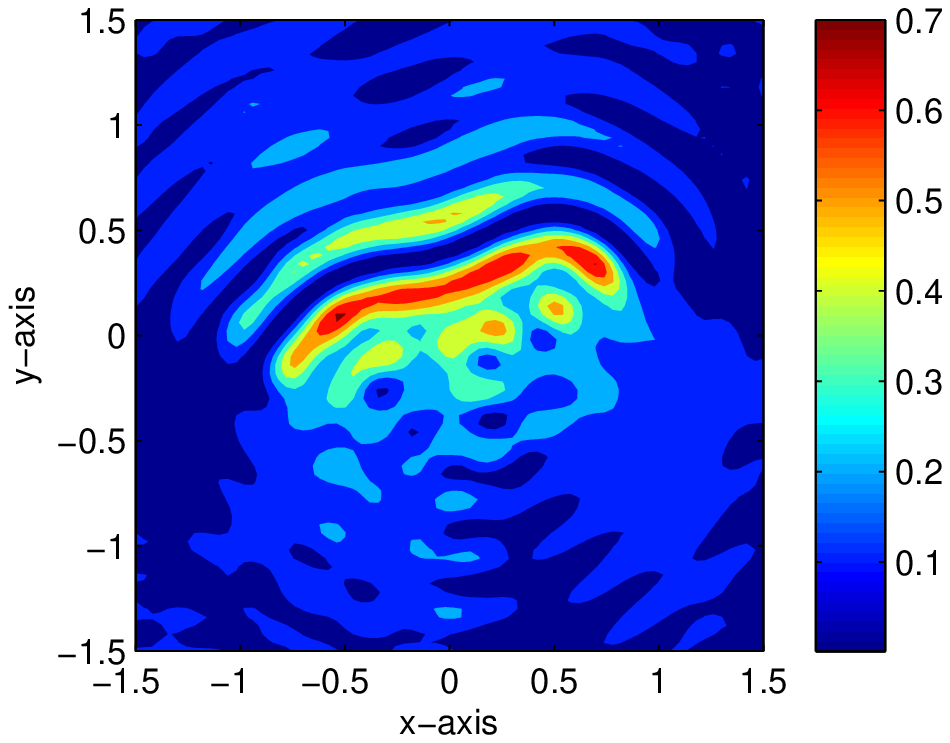}
\caption{\label{Result2}Same as Figure \ref{Result1} except the arc is $\Gamma_2$.}
\end{center}
\end{figure}

\section{Conclusion}\label{sec:5}
In this study, a subspace migration imaging algorithm was considered to image an arbitrary shaped smooth sound-hard arc modeled via a Neumann boundary condition (TE polarization in electromagnetics) in a two-dimensional full-view inverse acoustic problem. This is based on the factorization of a collected MSR matrix and the structure of singular vectors linked to the nonzero singular values. The results indicated that the imaging function can be expressed as the combination of the Bessel functions of the order $0$, $1$, and $2$ of the first kind, and this expression yields certain properties of the subspace migration.

The main subject of this paper is the imaging of sound-hard arc in homogeneous space. An extension to limited-aperture problem \cite{CAB,INS,KP} or the half-space problem \cite{AIL1,P-SUB5,PL2,PP2} will be an interesting research subject. Finally, we expect the analysis presented in this paper could be extended to three-dimensional problems \cite{AILP,GLCP,IGLP} and real-world application such as biomedical imaging \cite{HSM1,SKHV,YKKJP}.

\section*{Acknowledgement}
The author would like to acknowledge two anonymous referees for their precious comments. This research was supported by the Basic Science Research Program of the National Research Foundation of Korea (NRF) funded by the Ministry of Education (No. NRF-2017R1D1A1A09000547).

\appendix
\section{Derivation of Lemma \ref{Lemma}}\label{sec:A}
  The polar coordinate is considered: it is assumed that $\vx=(\cos\xi,\sin\xi)$, $\vz=(\cos\zeta,\sin\zeta)$, $\mx=r(\cos\phi,\sin\phi)$, and $\vx\ne\vz$.

  \begin{enumerate}
  \item It is assumed that $\vx\ne\vz$. Then, given that $\vt\cdot\vx=\cos(\theta-\xi)$, $\vt\cdot\vz=\cos(\theta-\zeta)$, and $\vt\cdot\mx=r\cos(\theta-\phi)$, the Jacobi-Anger expansion is applied as follows:
  \begin{equation}\label{JA}
    e^{iz\cos\phi}=J_0(z)+2\sum_{n=1}^{\infty}i^nJ_n(z)\cos(n\phi)
  \end{equation}
  yields
  \begin{equation}\label{formula1}
    \int_{\mathbb{S}^1}(\vt\cdot\vx)(\vt\cdot\vz)e^{ik\vt\cdot\mx}d\vt=\int_0^{2\pi}\cos(\theta-\xi)\cos(\theta-\zeta)\left(J_0(kr)+2\sum_{n=1}^{\infty}i^nJ_n(kr)\cos(n(\theta-\phi))\right)d\theta.
  \end{equation}
  Since (\ref{JA}) holds uniformly and the following expression holds:
  \[\cos(\theta-\xi)\cos(\theta-\zeta)=\frac12\bigg(\cos(2\theta-\xi-\zeta)+\cos(\xi-\zeta)\bigg),\]
  (\ref{formula1}) can be expressed as follows:
  \begin{align*}
    \int_{\mathbb{S}^1}(\vt\cdot\vx)(\vt\cdot\vz)e^{ik\vt\cdot\mx}d\vt=\frac12\int_0^{2\pi}\cos(2\theta-\xi-\zeta)J_0(kr)d\theta&+\sum_{n=1}^{\infty}i^nJ_n(kr)\int_0^{2\pi}\cos(2\theta-\xi-\zeta)\cos(n(\theta-\phi))d\theta\\
    +\frac12\int_0^{2\pi}\cos(\xi-\zeta)J_0(kr)d\theta&+\sum_{n=1}^{\infty}i^nJ_n(kr)\int_0^{2\pi}\cos(\xi-\zeta)\cos(n(\theta-\phi))d\theta.
  \end{align*}

  The application of an elementary calculus makes it  possible to easily observe that for all $n=1,2,\cdots$, the following expressions hold:
  \begin{equation}\label{term1}
    \int_0^{2\pi}\cos(2\theta-\xi-\zeta)d\theta=0\quad\mbox{and}\quad\int_0^{2\pi}\cos(\xi-\zeta)\cos(n(\theta-\phi))d\theta=0
  \end{equation}
  and
  \begin{equation}\label{term2}
    \int_0^{2\pi}\cos(\xi-\zeta)J_0(kr)d\theta=2\pi\cos(\xi-\zeta)J_0(kr).
  \end{equation}
  Finally, since (see \cite{GR})
  \begin{equation}\label{Equation}
  \int\cos(ax+b)\cos(cx+d)dx=\left\{\begin{array}{ccc}
  \displaystyle\medskip\frac{\sin[(a-c)x+b-d]}{2(a-c)}+\frac{\sin[(a+c)x+b+d]}{2(a+c)}&\mbox{if}&a^2\ne c^2\\
  \displaystyle\frac{x}{2}\cos(b-d)+\frac{\sin(2ax+b+d)}{4a}&\mbox{if}&a=c,
  \end{array}\right.
  \end{equation}
  we can observe the following expressions:
  \begin{equation}\label{term3}
  \int_0^{2\pi}\cos(2\theta-\xi-\zeta)\cos(n(\theta-\phi))d\theta=\left\{\begin{array}{ccl}\medskip0&\mbox{ if }&n\ne2\\
  \pi\cos(2\phi-\xi-\zeta)&\mbox{ if }&n=2.\end{array}\right.
  \end{equation}
  Therefore, by combining (\ref{term1}), (\ref{term2}), and (\ref{term3}), the following expressions are obtained:
  \begin{align*}
    \frac{1}{2\pi}\int_{\mathbb{S}^1}(\vt\cdot\vx)(\vt\cdot\vz)e^{ik\vt\cdot\mx}d\vt&=\frac12\cos(\xi-\zeta)J_0(kr)-\frac12\cos(2\phi-\xi-\zeta)J_2(kr)\\
    &=\frac12\cos(\xi-\zeta)J_0(kr)-\frac12\bigg(\cos(2\phi-\xi-\zeta)-\cos(\xi-\zeta)\bigg)J_2(kr)-\frac12\cos(\xi-\zeta)J_2(kr)\\
    &=\frac12(\vx\cdot\vz)\bigg(J_0(k|\mx|)-J_2(k|\mx|)\bigg)-\left(\frac{\mx}{|\mx|}\cdot\vx\right)\left(\frac{\mx}{|\mx|}\cdot\vz\right)J_2(k|\mx|).
  \end{align*}
  Hence, (\ref{Identity1}) is derived.
  \item Now, it is assumed that $\vx=\vz$. Then, the application of (\ref{JA}) leads to the following observation:
  \begin{align*}
\frac{1}{N}\sum_{n=1}^{N}(\vt_n\cdot\vx)^2e^{ik\vt_n\cdot\mx}&\approx\int_{\mathbb{S}^1}(\vt\cdot\vx)^2e^{ik\vt\cdot\mx}d\vt=\int_0^{2\pi}\cos^2(\theta-\xi)e^{ik r\cos(\theta-\phi)}d\theta\\
&=\int_0^{2\pi}\cos^2(\theta-\xi)\left(J_0(k r)+2\sum_{n=1}^{\infty}i^nJ_n(k r)\cos n(\theta-\phi)\right)d\theta\\
&=J_0(k r)\int_0^{2\pi}\cos^2(\theta-\xi)d\theta+2\sum_{n=1}^{\infty}i^nJ_n(k r)\int_0^{2\pi}\cos^2(\theta-\xi)\cos n(\theta-\phi)d\theta.
\end{align*}
Given that the following expression holds:
\[\int_0^{2\pi}\cos^2(\theta-\xi)d\theta=\frac{1}{2}\]
and given the following (\ref{Equation}),
\begin{align*}
\int_0^{2\pi}\cos^2(\theta-\xi)\cos n(\theta-\phi)d\theta&=\frac{1}{2}\int_0^{2\pi}\bigg(\cos2(\theta-\xi)+1\bigg)\cos n(\theta-\phi)d\theta=0
\end{align*}
for all $n\in\mathbb{N}$. Hence, the following expression is obtained (\ref{Identity2}).
\end{enumerate}

\bibliographystyle{elsarticle-num-names}
\bibliography{References}

\end{document}